\documentclass[12pt]{article}
\usepackage{amsmath,amsfonts,amsthm, amssymb}
\usepackage{subeqnarray,pdfsync}
\usepackage{pdfsync}
\usepackage{graphicx, epstopdf}
\usepackage{stmaryrd}
\usepackage{hyperref}
\usepackage{subeqnarray, cases}
\def\al{\alpha}
\def\be{\beta}
\def\ga{\gamma}
\def\de{\delta}

\def\Om{\Omega}    \def\om{\omega}

\newtheorem{claim}{Claim}

\newtheorem{theorem}{Theorem}
\newtheorem{maintheorem}{Theorem}

\newtheorem{corollary}[theorem]{Corollary}
\newtheorem{definition}[theorem]{Definition}

\newtheorem{lemma}[theorem]{Lemma}
\newtheorem{notation}[theorem]{Notation}
\newtheorem{proposition}[theorem]{Proposition}
\newtheorem{remark}[theorem]{Remark}

\theoremstyle{definition}

\usepackage{color, fancybox, empheq}

\newcommand{\comment}[1]{}

\newdimen\AAdi%
\newbox\AAbo%
\def\AArm{\fam0 }%\tenrm}%
\def\AAk#1#2{\setbox\AAbo=\hbox{#2}\AAdi=\wd\AAbo\kern#1\AAdi{}}%
\def\AAr#1#2#3{\setbox\AAbo=\hbox{#2}\AAdi=\ht\AAbo\raise#1\AAdi\hbox{#3}}%
\def\BBone{{\AArm 1\AAk{-.8}{I}I}}%

\newcommand {\CC}{{\mathcal C}}

\newcommand {\CM}{{\mathcal M}}

\newcommand{\disp}{\displaystyle}
\newcommand{\eps}{\varepsilon}
\newcommand{\8}{\infty}

\def\m1{{-1}}

\newcommand{\ninf}{{n\rightarrow+\8}}
\newcommand{\ol}{\overline}
\def\S{\Sigma}
\def\s{\sigma}
\def\l{\lambda}
\def\supp{\mbox{supp}\,}

\def\ga{\gamma}

\newcommand{\N}{\mathbb{N}}
\newcommand{\R}{\mathbb{R}}

\newcommand{\Z}{\mathbb{Z}}

\def\Q{{\mathbb Q}}

\title{On the speed of convergence of the pressure function at zero temperature}

\author{R. Leplaideur\thanks{ISEA \& LMBA UMR6205, Université de la Nouvelle Calédonie}  }

\begin{document}

\maketitle

\begin{abstract}
 We prove here that the pressure function cannot converge to the limit entropy at zero temperature faster than some exponential rate. Furthermore, we characterize this limit rate via an expression involving the Peierls barriers between the irreducible components of the Aubry set. This extends and completes results from \cite{Leplaideur-Mengue} and \cite{KuchQuas}. In the first one, an exact exponential speed of convergence was proved, under the assumption that the Aubry set is a subshift of finite type. In the later one, a rate was given but without interpretation in term of Thermodynamical quantities of the system. 
\end{abstract}



%%%%%%%%%%%%
%%%%%%%%%%%%

\section{Introduction}
%%%%%
\subsection{Background and settings}

In this paper we show that the pressure function converges at most exponentially fast to the limit entropy as the temperature goes to 0. 
We remind that, given a dynamical system $(X,T)$ and a potential $A:X\to \R$, the pressure function is defined by 
$$P(\be):=\sup_{\mu\ T-inv}\left\{h_{\mu}(T)+\be\cdot\int A\,d\mu\right\},$$
where $h_{\mu}(T)$ is the Kolmogorov entropy and $\be$ is a real parameter. In statistical mechanics $\be$ represents the inverse of the temperature. 
It is known that $P(\be)$ goes to a non-negative value $h$ as $\be$ goes to $+\8$. This value $h$ is the topological entropy of the Aubry set $\Om$, and also of the Mather set for the potential $A$, $\CM_{a}$. The Mather set is the union of the supports of the \emph{ maximizing measures}, that is, any $T$-invariant probability measure $\mu$ such that 
$$\int A\,d\mu=m(A)=:\sup_{\nu\ T-inv}\int A\,d\nu.$$

In the historical unpublished paper \cite{Conze-Guivarch} it was proved that $P(\be)$ goes to $h$ faster than $1/\be^{2}$. In \cite{KuchQuas}, it is proved that the convergence cannot be faster than exponential, but there is no identification of the exponential rate (in term of quantities linked to the thermodynamics). 
In \cite{Leplaideur-Mengue} it is proved, under the assumption that the Aubry set is a subshift of finite type,  that the convergence  occurs at exponential rate, and the rate is characterized (in term of quantities linked to the thermodynamics). 

Here, we prove that in the general case, the convergence cannot be faster than exponential and, as in \cite{Leplaideur-Mengue}, we show that the minimum rate is linked to thermodynamic quantities. 

%\subsection{Settings}
\subsection{Settings and results}
In the following  $X\subseteq \{1,...,D\}^{\mathbb{N}}$ is a subshift of finite type given by an aperiodic matrix. 
A point in $X$ will be referred to as an infinite sequence or an infinite words $x=x_{0}x_{1}\ldots$, where the  $x_{i}$'s are the \emph{digits} (in $ \{1,\ldots, D\}$). 
A cylinder $[\om_{0}\ldots\om_{n}]$ is the set of words $x=x_{0}x_{1}\ldots$ such that $x_{i}=\om_{i}$ for $i\le n$. It will also be denoted as a $(n+1)$-cylinder.

$X$ equipped with the metric $d$ defined by 
$$d(x_{0}x_1x_2x_3\ldots,y_{0}y_1y_2y_3\ldots) = 2^{-\min\{i\,|\,x_i\neq y_i\}}$$
is a compact space. 
The shift map $\sigma:X \to X$ is defined by
$$x_{0}x_1x_2x_3\ldots = x_{1}x_2x_3x_4\ldots.$$

We consider $A:X\to \R$, Lipschitz continuous. The set of $\s$-invariant probabilities is denoted by $\CM(\s)$. We set 
$$m(A):=\max_{\mu\in\CM(\s)}\int A\,d\mu.$$
Any measure realizing this supremum is called $A$-maximizing or maximizing for $A$. 
We denote by $\CM_{max}(A)$ the set of $A$-maximizing measures. We remind (see \cite{BLL} and below) that we always may assume that $A$ is  non-positive and $m(A)=0$. From now on, we assume that this holds. 

We remind that $S_{n}(A)$ stands for the Birkhoff sum $A+\ldots +A\circ \s^{n-1}$.

\begin{definition}
\label{def-maneaubrymather}
The \textbf{Ma\~n\'e potential} associated to $A$ is defined by 
	$$S(x,y):=\lim_{\epsilon \rightarrow 0}\left[\sup\{S_n(A)(z),\ \s^{n}(z)=y,\ d(x,z)<\eps\}\right],$$
the \textbf{Aubry set} of $A$ is defined by
	$$\Omega:=\{x\in X\, |\, S(x,x) = 0\}$$ 
	%and} \textbf{Mather set}  is 
	%$$\mathcal{M}_{a} = \cup_{\mu \in \mathcal{M}_{\max}(A)}\operatorname{supp}(\mu).$$
\end{definition}

Note that $A\le 0$ yields that  $S(.,.)$ is non-positive.

We remind that the Mather set $\CM_{a}$ is included into the Aubry set $\Om$. 
We emphasize that the Mather set is never empty, since $\mu\mapsto \int A\,d\mu$ is continuous over the compact set $\CM(\s)$. Hence the Aubry set is neither non-empty.

We remind (see also Lemma \ref{lemaManebon}):
\begin{itemize}
\item  $S(x,y)+S(y,z)\le S(x,z)$ for every $x$, $y$, $z$ in $X$;
\item for any fixed $x\in X$, the map $y\mapsto S(x,y)$ is Lipschitz continuous;
\item $x\mapsto S(x,y)$ is upper semi-continuous. 
\end{itemize}

The next definition aims to define irreducible components of the Aubry set. The proof of the proposition part is done later (see Lemma \ref{lem-equivrel}).

\begin{definition}[{\bf and Proposition}]
\label{def-irreduaubry}
The relation $S(x,y)+S(y,x)=0$ defines an equivalence relation. Each equivalence class is called an irreducible component of the Aubry set. 
\end{definition}

%\begin{proof}TO WRITE LATER
%By definition $x\in \Om$ if and only if $S(x,x)=0$, which yields $x\sim x$. 
%Furthermore 
%$$S(x,y)+S(y,x)=S(y,x)+S(x,y),$$
%hence $x\sim y$ yields $y\sim x$. 
%Finally, if $x\sim y$ and $y\sim z$ hold, 
%\begin{eqnarray*}
%0\ge S(x,z)+S(z,x)&\ge& S(x,y)+S(y,z)+S(z,y)+S(y,x)\\
%&=&S(x,y)+S(y,x)+S(y,z)+S(z,y)=0+0+0+0=0.
%\end{eqnarray*}
%
%Hence, $x\sim z$ holds.
%\end{proof}

It is proved below (see Lemma \ref{lem-irreclosed}) that irreducible components are closed sets. Therefore, they are compact sets and two different irreducible components are at positive distance one from the other.

\begin{definition}
\label{def-costs2}
Let $\Om_{i}$ and $\Om_{j}$ be irreducible components (possibly $i=j$). Then we set 
$$S^{ext}(j,i):=\inf_{x\in \Om_{i}}\sup\{S(y,z)+A(z),\s(z)=x,\ z\notin \Om_{i},\ y\in \Om_{j}\}.$$
\end{definition}

%If $j\neq i$, note that $S^{ext}(j,i)=$

We claim (see Lemma \ref{lem-finitcosts3}) that each $S^{ext}(j,i)$ is a non-positive real number. Furthermore, Lemma \ref{lemaManebon} will show that the $\sup$ can be replaced by  a $\max$.
%\newpage

%
%
%\begin{definition}
%\label{def-costs}
%For $x\in \Omega$, if $x$ belongs to the irreducible component $\Omega_{i}$, then,  we set 
%$$S^{ext}(y,x):=\lim_{\epsilon \rightarrow 0}\left[\sup\{S_n(A)(z),\ \s^{n}(z)=x,\ d(y,z)<\eps,\ z\notin \Omega_{i}\}\right],$$
%
%Let $\Omega_{j}$ be an irreducible component for $\Omega$. Then, we set 
%$$S^{ext}(\Omega_{j},x):=\sup\{S^{ext}(y,x), \ y\in \Omega_{j}\}.$$
%\end{definition}
%In spirit, the value $S^{ext}(\Omega_{j},x)$ represents the minimal cost to go (backward) from $x$ to the irreducible component $\Omega_{j}$, and going out of the irreducible component of $x$. If $\Omega_{j}\neq \Omega_{i}$, then the condition $z\notin\Omega_{i}$ is irrelevant since two irreducible components are at positive distance. 
%
%We claim (see Proposition \ref{prop-costfinis}) that for every $x$ and every $\Omega_{j}$, $S^{ext}(\Om_{j},x)$ is finite. 
If $\Om_{i_{1}},\ldots \Om_{i_{n}}$ are (different) irreducible components, then we denote by $\l(i_{1},\ldots, i_{n})$ the unique eigenvalue (see \cite{BCOQ}) for the $n\times n$ matrix with entries $S^{ext}(j,i)$ and $i,j=i_{1},\ldots, i_{n}$ within the {\bf max-plus formalism}.

We remind that for the $n\times n$ matrix with entries $a_{ij}>-\8$, the eigenvalue $\l$  within the max-plus formalism is the maximal mean value of the entries along the cycles: 
$$\l=\max_{i_{1},i_{2},\ldots ,i_{k} \text{ all different}}\dfrac{a_{i_{1}i_{2}}+a_{i_{2}i_{3}}+\ldots +a_{i_{k-1}i_{k}}+a_{i_{k}i_{1}}}k.$$

We remind that the pressure function $\be\mapsto P(\be)$ is convex\footnote{this immediately follows from definition.} with $P'(\be)=\int A\,d\mu_{\be}$, where $\mu_{\be}$ is the unique equilibrium state for $\be.A$. Therefore, $P(\be)$ is non-increasing (in $\be$) and is bounded from below by $h_{\mu}$ where $\mu$ is any maximizing measure. A classical argument shows that $P(\be)$ converges to $h$, which is the maximal entropy among maximizing measures. 

The main question we are interested in is at which speed does the convergence occur. 

\begin{maintheorem}\label{theoA}
We get a minimal exponential rate for the convergence of the pressure:
$$\hskip-3cm\liminf_{\be\to+\8}\dfrac1\be\log(P(\be)-h)\geq \sup\{\l(i_{1},\ldots, i_{n}),\ n\in \N^{*},\  \Om_{i_{k}}\text{ are irreducible components with maximal entropy}\}.$$
\end{maintheorem}

%%%%%
%%%%%%
\subsection{On the structure of the Aubry set}
Subshifts of finite are perfectly well understood. We refer to \cite[Ex. 1.9.4-]{katok-hasselblatt} for their decompositions in irreducible components. 
This has been used in \cite{Leplaideur-Mengue}. 

We emphasize here at least two important points where the assumption ``the Aubry set is a subshift of finite type'' is relevant:
\begin{enumerate}
\item Equality or inequality on the speed of convergence. Equality comes from the existence of conformal measures associated to the measure of maximal entropy in each irreducible component with maximal entropy. On the other hand, for general shift, there is no general theory. This explains why we only get\footnote{at least for the moment.} inequalities (see below) in Theorem \ref{theoA}. 
\item Influence of the components with minor entropy. In the case $h>0$, it is proved\footnote{still under the assumption that the Aubry set is a subshift of finite type.} in \cite{Leplaideur-Mengue} that components of the Aubry with entropy strictly less than $h$ do not play any role in the speed of convergence. The proof deeply uses the fact that in the case the Aubry set is a subshift of finite type, it only has finitely irreducible components. In the case $\Om$ is a general subshift, it may have uncountably many irreducible components. We provide an example below. 
\end{enumerate}

We warmly thank Sebastien Ferenczy for having given us the references we needed to construct the example. 

We consider some small closed interval $[a,b]$ and then consider the union of Sturmian shifts with angle $\al\in [a,b]$. We refer to \cite{Morse-Hedlund} for a general description of these Sturmian trajectories. 

More precisely, $\Om_{\al}$ is the subshift whose language correspond to all the words that appear, considering a line passing to $(0,0)$ in $\R^{2}$, with slope $\al$, and writing 0 when the line crosses an horizontal line in the net defined by $\Z^{2}$, and 1 if it crosses a vertical line. 

It is known that each $\Om_{\al}$ is uniquely ergodic which implies that $\Om:=\disp\bigcup_{\al\in[a,b]}\Om_{\al}$ has uncountably many irreducible components, some of them being periodic orbit (iff $\al\in\Q$). 

It is known (see \cite[Corollary 18]{Mignosi}) that the number $A_{n}$ of all the words of length $n$ appearing in some Sturmian subshift is equivalent to $\dfrac{n^{3}}{\pi^{2}}$ as $n$ goes to $+\8$. This immediately yields that $\Om$ has zero entropy. 

Now consider $A:=-d(.,\Om)$. Then $\Om$ is the Aubry set for $A$ and it has infinitely many irreducible components.

%%%%%
%%%%%
\subsection{Plan of the paper}
In Section \ref{sec-tech} we give technical results. In particular we remind basic results on the calibrated subactions and the Ma\~né Potential. 
In Section \ref{sec-proof} we do the proof of the Theorem \ref{theoA}. The main point is Proposition \ref{prop-minoprincipale} where we prove a key inequality.

%%%%%%%%
%%%%%
%%%%%%
\section{Technical results}\label{sec-tech}

%%%%%%%%
%%%%%%%
\subsection{On the Ma\~né potential}

\subsubsection{Technics to compute/estimate it}

For $\eps>0$ we set 
$$S^{\eps}(x,y):=\sup{\{S_{k}(A)(z),\ \s^{k}(z)=y,\ d(x,z)<\eps\}}.$$
Then $S(x,y)=\lim_{\eps\to0}S^{\eps}(x,y)$. 

\begin{lemma}
\label{lem-techmane}
Let $x$ and $y$ be in $X$. Let $\ol\xi_{n}$ be points such that $\ol\xi_{n}\to x$ as $n$ goes to $+\8$, and for any $n$ there exists $k_{n}$ with $k_{n}\to+\8$ as $n\to+\8$ such that $\s^{k_{n}}(\ol\xi_{n})=y$. Then, 
$$S(x,y)\ge\limsup_{n\to+\8}S_{k_{n}}(A)(\ol\xi_{n}).$$
\end{lemma}
\begin{proof}
Pick $\eps>0$. Let $N$ be such that for every $n>N$, $d(x,\ol\xi_{n})<\eps$. Then 
$$S^{\eps}(x,y):=\sup{\{S_{k}(A)(z),\ \s^{k}(z)=y,\ d(x,z)<\eps\}}\ge S_{k_{n}}(A)(\ol\xi_{n})$$
holds as soon as $n>N$ holds.
This yields, for every $\eps$, $\disp S^{\eps}(x,y)\ge\limsup_{n}S_{k_{n}}(A)(\ol\xi_{n})$. Hence doing $\eps\downarrow 0$ we get the result. 
\end{proof}

%\begin{lemma}
%\label{lemtechmane2-test}
%Let $x$ be in some irreducible component $\Om_{i}$. Then, there exists a sequence $(\xi^{n})$ of points and integers $k_{n}$ such that 
%\begin{enumerate}
%\item $k_{n}\to+\8$ as $n\to+\8$,
%\item $\xi^{n}\to_{\ninf}x$,
%\item $\s^{k_{n}}(\xi^{n})=x$, 
%\item $S_{k_{n}}(A)(\xi^{n})\to0$ as $n\to+\8$.
%\end{enumerate}
%\end{lemma}
%\begin{proof}
%We first do the proof in the case $x$ is non-periodic. 
%By definition of $\Om$, $S(x,x)=0$. This means that 
%$$\lim_{\eps\to0}\sup\{S_{n}(A)(z),\ d(x,z)<\eps, \s^{n}(z)=x\}=0.$$
%Then for $n>0$,  for any sufficiently small $\eps$, one can find $z=:\xi^{n}$ such that 
%$$0\ge S_{m}(A)(z)\ge -\dfrac 1n$$
% and $d(x,z)<\eps$. As $x$ is non-periodic, one can always construct the sequence by induction and pick a sufficiently small $\eps$ at step $n+1$ such that the $m=:k_{n+1}$ at sep $n+1$ is strictly bigger than $k_{n}$. 
%
%If $x$ is $k$-periodic, $\xi^{n}=x$ and $k_{n}=nk$ answer the problem.  
%\end{proof}
%

%%%%%
\subsubsection{Mains properties and one immediate consequence}

Let us first recall a classical result for the Ma\~né potential:
	
	\begin{lemma}[See \cite{Leplaideur-Mengue}] \label{lemaManebon} Let $S$ be the Ma\~{n}\'{e} potential. We have:
	\begin{enumerate}
\item  $S(x,y)+S(y,z)\le S(x,z)$ for every $x$, $y$, $z$ in $X$;
\item for any fixed $x\in X$, the map $y\mapsto S(x,y)$ is Lipschitz continuous;
\item if $x\sim y$, then $S(x,.)=S(y,.)$;
\item if $\s(x)=y$, then $S(x,y)=A(x)$.
\end{enumerate}
\end{lemma}

%
%\begin{proof} For items 1 and 2 see \cite{BLL}. 
%
%In order to prove {\it 3}, note that $S(x,y)+S(y,x)=0$ and $S(.,.)\le 0$ yields 
%$$S(x,y)=S(y,x)=0.$$ 
%Hence, for any $z\in X$,
%	\[S(x,z) \stackrel{1.}{\geq} S(x,y)+S(y,z) = S(y,z)  \stackrel{1.}{\geq} S(y,x)+S(x,z) = S(x,z).\]
%	
%
%\medskip
%The proof of item {\it 4} has two steps. First, note that if $	z$ is such that $\s^{n}(z)=y$ and $d(x,z)<\eps$, then 
%$$S_{n}A(z)\le A(z)\le A(x)+Cd(x,z) \le A(x)+C\eps.$$
%{This yields for any positive $\eps$:
%$$S^{\eps}(x,y)=\sup\{S_n(A)(z),\ \s^{n}(z)=y,\ d(x,z)<\eps\}\le { A(x)+C\eps}.$$
%Doing $\eps\downarrow0$ we get $S(x,y)\le A(x)$. 
%
%\medskip
%On the other hand, as $\s(x)=y$, 
%$$\sup\{S_n(A)(z),\ \s^{n}(z)=y,\ d(x,z)<\eps\}\ge A(x),$$
%which yields the other inequality. }
%\end{proof}

\begin{lemma}
\label{lem-equivrel}
$S(x,y)+S(y,x)=0$ defines an equivalence relation. 
\end{lemma}
\begin{proof}
By definition $x\in \Om$ if and only if $S(x,x)=0$, which yields $x\sim x$. 
Furthermore 
$$S(x,y)+S(y,x)=S(y,x)+S(x,y),$$
hence $x\sim y$ yields $y\sim x$. 
Finally, if $x\sim y$ and $y\sim z$ hold, 
\begin{eqnarray*}
0\ge S(x,z)+S(z,x)&\ge& S(x,y)+S(y,z)+S(z,y)+S(y,x)\\
&=&S(x,y)+S(y,x)+S(y,z)+S(z,y)=0+0+0+0=0.
\end{eqnarray*}

Hence, $x\sim z$ holds.
\end{proof}

We remind that irreducible components of $\Om$ are the classes of equivalences in $\Om$. 

\begin{notation}
 For $x$ in $\Om_{i}$ we shall set $S(\Om_{i},\cdot)$ for $S(x,\cdot)$. 
\end{notation}
An immediate consequence is that in the definition of $S^{ext}(j,i)$ we can replace $\sup$ by $\max$, since any $x$  admits finitely preimages (of order 1) and $S(y,z)$ does not depend on the choice of $y\in \Om_{j}$.

%%%%%%%%
\subsubsection{Irreducible components and costs}

\begin{lemma}
\label{lem-AonOm}
The potential $A$ is equal to 0 on $\Om$. 
\end{lemma}
\begin{proof}
By definition $x$ belongs to $\Om$ if and only if $S(x,x)=0$. Hence we get 
$$ 0\ge A(x)\geq S(x,x)=0.$$
\end{proof}

\begin{lemma}
\label{lem-siinul}
If $\Om_{i}$ is an irreducible component, then $S^{ext}(i,i)<0$.
\end{lemma}
\begin{proof}
By our assumption on $A$, $S^{ext}(i,i)$ is non-positive. If $x$ is in $\Om_{i}$ and $\s(z)=x$ with $z\notin \Om_{i}$, then two cases occur:

\begin{enumerate}
\item either $A(z)<0$, and $S(\Om_{i},z)+A(z)<0$,
\item or $A(z)=0$ and $S(\Om_{i},z)<0$. 
\end{enumerate}
Indeed, in the later case if $S(\Om_{i},z)=0$, we get $S(x,z)=0$ (since $x\in \Om_{i}$) and $A(z)=S(z,x)=0$. This yields 
$$S(x,z)+S(z,x)=0,$$
hence $z\sim x$ which does not hold. As $x$ has finitely many preimages 
$$S^{ext}(i,i)\le \max\{S(\Om_{i},z)+A(z),\s(z)=x,\ z\notin\Om_{i}\}<0.$$

\end{proof}

\begin{lemma}
\label{lem-irreclosed}
Irreducible components of the Aubry set are closed sets. 
\end{lemma}
\begin{proof}
Let us pick $x$ and $y_{n}$ such that $x\sim y_{n}$ for every $n$. We assume that $(y_{n})$ converges to $y$. 
We prove that $y\sim x$. 

We remind that $S(x,.)$ is Lipschitz continuous. Furthermore, $\sim y_{n}$ means $S(x,y_{n})+S(y_{n}x)=0$, and since $S$ is non-positive it yields (for every $n$)
$$S(x,y_{n})=0=S(y_{n},x).$$
Then, doing $n\to+\8$ in $S(x,y_{n})=0$ yields $S(x,y)=0$. 

\medskip
Let us now prove that $S(y,x)=0$ holds. 

Let $\al$ be positive. Let $\eps$ be positive. 
We choose $N$ such that $d(y_{N},y)<\eps/2$. As $S(y_{N},x)=0$, we get 
$$\lim_{\rho\to0}\sup\{S_{n}(A)(z),\ \s^{n}(z)=x,\ d(z,y_{N})<\rho\}=0.$$
We can find $\rho_{0}>0$ such that for every $0<\rho<\rho_{0},$
$$S^{\rho}(y_{N},x)=\sup\{S_{n}(A)(z),\ \s^{n}(z)=x,\ d(z,y_{N})<\rho\}>-\frac\al2.$$
Let us pick $\rho>0$, such that $\rho<\rho_{0}$ and $\rho<\dfrac\eps2$ hold. One can find $z$ and $n$ such that $\s^{n}(z)=x$, $d(z,y_{N})<\rho$ and 
$$S_{n}(A)(z)\ge -\al.$$
Then $d(z,y)<\eps$, $\s^{n}(z)=x$ and $S_{n}(A)(z)>-\al$. 
This yields $$S^{\eps}(x,y)=\sup\{S_{n}(A)(z),\ \s^{n}(z)=x,\ d(z,y)<\eps)\}>-\al.$$
This holds for every $\eps$, hence $S(y,x):=\lim_{\eps\to0}S^{\eps}(y,x)$ is greater than $-\al$. 
This also holds for every $\al$, and since $S(y,x)$ is non-positive, we get $S(y,x)=0$.
\end{proof}

\begin{lemma}
\label{lem-finitcosts3}
For any $i$ and $j$, $S^{ext}(i,j)$ is a finite non-positive real number.
\end{lemma}
\begin{proof}
By our assumption on $A$, $S^{ext}(i,j)$ is non-positive. 
Then, for any $x\in \Om_{i}$, and for any $y$, 
$$\sup_{z\in\Om_{i}}S(z,y)\ge S(x,y).$$
The map $y\mapsto S(x,y)$ is Lipschitz continuous hence bounded on $X$. Therefore 
$$S^{ext}(i,j)=\inf_{y'\in\Om_{j}}\sup_{z\in \Om_{i}}\{S(z,y),\ y\notin\Om_{i}, \s(y)=y'\}\ge -||S(x,\cdot)||_{\8}>-\8.$$
\end{proof}

\begin{lemma}
\label{lem-bonnesuite}
Let $y$ be in $\Om_{i}$ and $x$ be in $X$. There exists a sequence of points $(\xi^{n})$  and integers $k_{n}$ such that 
\begin{enumerate}
\item $\xi^{n}\to y$ as $n\to+\8$, 
\item $\s^{k_{n}}(\xi^{n})=x$,
\item $S(y,x)=\lim_{\ninf}S_{k_{n}}(A)(\xi^{n})$.
\end{enumerate}
\end{lemma}
\begin{proof}
We get $S(y,x)=\lim_{\eps\to0}S^{\eps}(y,x)$. Set $\eps:=\dfrac1n$ and consider $\xi^{n}$ and $k_{n}$ such that 
$$d(y,\xi^{n})\le \dfrac1n,\quad S^{\frac1n}(y,x)\ge S_{k_{n}}(A)(\xi^{n})\ge S^{\frac1n}(y,x)-\dfrac1n, \quad x=\s^{k_{n}}(\xi^{n}).$$
Then, 
$\xi^{n}\to y$ as $\ninf$
\end{proof}

%%%%%%%%
%%%%%%%

\subsection{Transfer Operator, eigenfunctions, calibrated subactions}
For an irreducible subshift of fintie type (not necessarily $X$), say $Y$, and $B:Y\to \R$ Lipschizt-continuous, 
The transfer operator $L_B$ is defined by
\[L_{B}(\psi)(x) = \sum_{\sigma(z) =x}e^{B(z)}\psi(z).\]
We refer to \cite{Bowen} for complete study of this operator. It acts on $\CC(X)$ and on the space  $\CC^{0+1}(X)$ of Lipschitz continuous functions. Its spectral radius $\l_{B}$ (for $|\quad |_{\8}$-norm) is a single dominating eigenvalue on $\CC^{0+1}(X)$. It also turns out to be equal to $e^{P}$, where $P$ is the pressure associated to $B$ and for the dynamical system $(Y,\s)$. The associated 1-dimensional  eigen-space is $span(H_{B})$, where $H_{B}$ is Lipschitz continuous, positive and uniquely determined up to some  normalization. The dual operator $L_{B}^{*}$ for the $|\quad |_{\8}$-norm acts on the set of measures and $\nu_{B}$ is the unique probability satisfying $L_{B}^{*}(\nu_{B})=\l_{B}\nu_{B}$. It is referred to as the \emph{eigenmeasure} or the \emph{conformal measure}.

In the case $Y=X$ and $B:=\be.A$, then the family of functions $\dfrac1\be\log H_{\be.A}$ is equicontinuous (actually Lipschitz continuous). Any accumulation point is a \emph{calibrated subaction}. We remind that a calibrated subaction for $A$ is a Lipschitz continuous function $V:X\to\R$ satisfying for every $x\in X$ 
$$\max_{\sigma(z)=x} [A(z) + V(z)] = V(x).$$
Furthermore we have 

\begin{lemma}[See \cite{BLT,Leplaideur-Mengue}]\label{lem-mane}
For any calibrated subaction $V$ we have $\forall\, x_0 \in X$
$$V(x_0) = \sup_{y\in \mathcal{M}_{a}} \left\{V(y) +  S(y,x_0)\right\}=\sup_{a\in \Omega} \left\{V(y) +  S(y,x_0)\right\}.$$
\end{lemma}

\begin{corollary}\label{coro-V(sigma)}
If $V$ is a calibrated subaction for $A$ then $V$ is constant on each set irreducible component $\Om_i$ of $\Om$.	
\end{corollary}

\begin{proof} 
	Let $x,y \in \Om_i$ and $V$ be a calibrated subaction.
	As $S(x,y)=S(y,x) = 0$,  applying Lemma \ref{lem-mane} we have
	\[ V(x) =\sup_{z\in \Omega}[S(z,x)+V(z)] \geq S(y,x) +V(y) = V(y).\]
and	\[V(y) = \sup_{z\in \Omega}[S(z,y)+V(z)] \geq S(x,y) +V(x) = V(x).\]
\end{proof}

\begin{notation}
 We shall write $V(\Om_{i})$.
\end{notation}

%%%%%%
\subsubsection{On the hypothesis on $A$}
For a general $A$ one get 
$$e^{P(\be.A)}H_{\be}(A)(x)=\sum_{y,\ \s(y)=x}e^{\be. A(y)}H_{\be}(y).$$
Doing $\dfrac1\be\log$ in this equality and then $\be\to+\8$, yields 
$$m(A)+V(x)=\max_{y}\{A(y)+V(y)\}.$$
Thus for any $y$ we get 
$$A(y)+V(y)-V\circ \s(y)-m(A)\le0.$$
Hence, exchanging the general $A$ by $A+V-V\circ \s-m(A)$, we get a new potential, equal to $A$ up to a co-boundary plus a constant, that is always non-positive and which maximal integral with respect to any invariant measure in 0. 

%%%%%%%%%%%
%%%%%%%%
\subsection{A kind of Laplace method result}
\begin{lemma}
\label{lem-laplace}
Let $(E,\de)$ be a metric space. Let $\varphi_{1},\ldots \varphi_{k}$ be continuous from $E$ to $\R$. Let $(\phi_{n, i})$ be a sequences of continuous functions from $E$ to $\R$ which uniformly converge to $\phi_{i}$. 
Let $(x_{n})$ be a sequence which converges to $x$ in $E$.
Then, $\disp \dfrac1n\log\left( \sum_{i=1}^{k}e^{n.\varphi_{i}(x_{n})+n\phi_{n,i}(x_{n})}\right)$ converges to $\max_{i}\{\varphi_{i}(x)+\phi_{i}(x)\}$. 
\end{lemma}
\begin{proof}
Let us  set $\l_{i}:=\varphi_{i}(x)+\phi_{i}(x)$,  $\l:=\max_{i} \l_{i}$.% and $\l'=\max\{\l_{j}, \l_{j}<\l\}$. 
Let $I:=\{i,\ \l_{i}=\l\}$.% and $J:=\{1,\ldots, k\}\setminus I$. 

Let $0<\eps$. Let $N$ be such that for any $n\ge N$, for any $i$, $\disp ||\phi_{n,i}-\phi_{i}||_{\8}<\eps$.  

Let $N'$ be such that for every $n\ge N'$ and for any $i$, $\disp |\varphi_{i}(x_{n})-\varphi_{i}(x)|<\eps,$ and $\disp  |\phi_{i}(x_{n})-\phi_{i}(x)|<\eps,$

\medskip
For $n\ge \max(N, N')$ we get 

\begin{eqnarray*}
\sum_{i=1}^{k}e^{n.\varphi_{i}(x_{n})+n\phi_{n,i}(x_{n})}&\ge & \sum_{i\in I}e^{n.\varphi_{i}(x_{n})+n\phi_{n,i}(x_{n})}\\
&\ge & \sum_{i\in I}e^{n.(\varphi_{i}(x)-\eps)+n(\phi_{i}(x_{n})-\eps)}\\
&\ge & \sum_{i\in I}e^{n.(\varphi_{i}(x)-\eps)+n(\phi_{i}(x)-2\eps)}\\
&\ge& e^{n(\l-3\eps)}.
\end{eqnarray*}
This yields for any $\eps$, $\disp\liminf_{\ninf}\dfrac1n\log\left(\sum_{i=1}^{k}e^{n.\varphi_{i}(x_{n})+n\phi_{n,i}(x_{n})}\right)\ge \l-3\eps$.

On the other hand we get
\begin{eqnarray*}
\sum_{i=1}^{k}e^{n.\varphi_{i}(x_{n})+n\phi_{n,i}(x_{n})}
&\le & \sum_{i=1}^{k}e^{n.(\varphi_{i}(x)+\eps)+n(\phi_{i}(x_{n})+\eps)}\\
&\le & \sum_{i=1}^{k}e^{n.(\varphi_{i}(x)+\eps)+n(\phi_{i}(x)+2\eps)}\\
&\le& k. e^{n(\l+3\eps)}.
\end{eqnarray*}
This yields  for any $\eps$, $\disp\limsup_{\ninf}\dfrac1n\log\left(\sum_{i=1}^{k}e^{n.\varphi_{i}(x_{n})+n\phi_{n,i}(x_{n})}\right)\le \l+3\eps$.

\end{proof}

\section{Proof}\label{sec-proof} 
In all that part, we shall consider a sequence $(\be_{n})$ such that $\disp \dfrac1{\be_{n}}\log\left(P(\be_{n})-h\right)$ converges to $\ga$. For simplicity we shall always write $\be\to+\8$ even if we actually only consider $\be_{n}\to+\8$.

%%%%%%
%%%%%
\subsection{Neighborhoods}

We pick some irreducible component, say $\Om_{i}$, with entropy $h_{i}$ equal to the maximal entropy $h$. We fix some small positive real number $\eps_{0}$ and then, consider a small neighborhood $B(\Om_{i},\eps_{0})$, that is the points at distance to $\Om_{i}$ lower than $\eps_{0}$. 

Then we denote by $\S_{i,\eps_{0}}$ the maximal invariant set in $B(\Om_{i},\eps_{0})$. It is non-empty since it contains $\Om_{i}$ and is a subshift of finite type. It can be decomposed in finitely many  irreducible components, and for simplicity we shall assume that $\S_{i,\eps_{0}}$ is indeed irreducible. We shall discuss this assumption later. 

Then, we denote by $\nu_{\eps_{0}}$ the unique conformal measure for the potential $B\equiv 0$ in $\S_{i,\eps_{0}}$. In the case $\S_{i,\eps_{0}}$ is not irreducible any irreducible component (they are finitely many) contains such a conformal measure.

\medskip
We then do the same work for a smaller neighborhood $B(\Om_{i},\eps)$ with $\eps<\eps_{0}$, with associated subshift  of finite type $\S_{i,\eps_{0}}$ and conformal measure $\nu_{\eps}$ (for the potential constant equal to zero). 

\begin{lemma}
\label{lem-accunueps}
Any accumulation point for the weak* topology $\nu_{1/\8}$ for $\nu_{\eps}$ as $\eps$ goes to 0 has support in $\Om_{i}$.
\end{lemma}
\begin{proof}
For $x\notin\Om_{i}$, $d(\{x\},\Om_{i})$ is positive, thus bigger than some $2.\eps_{1}$. For $\eps<\eps_{1}$, $x\notin B(\Om_{i},\eps)$, then $x$ is not in $\S_{i,\eps}$. Furthermore $\S_{i,\eps}$ contains the support for $\nu_{\eps}$. 

As $B(x,\eps_{1})$ is a clopen set, $\BBone _{B(x,\eps_{1})}$ is a continuous function. Hence for any $\eps<\eps_{1}$, 
$$\nu_{\eps}(B(x,\eps_{1}))=\int \BBone_{B(x,\eps_{1})}\,d\nu_{\eps}=0.$$
This yields $\nu_{1/\8}(B(x,\eps_{1}))=0$, which means that $x$ does not belong to $\supp\nu_{1/\8}$. 
\end{proof}

\begin{remark}
\label{rem-nulimite}
We point out one of the main problems we shall deal with. There is no way to know if $\nu_{1/\8}$ has full support or no.  
$\blacksquare$\end{remark}

In particular, we will need estimates on points $x\in \Om_{i}$ with preimages outside $\Om_{i}$. As we cannot ensure such a set with positive $\nu_{1/\8}$-measure exists, we need to get round this difficulty.

\begin{lemma}
\label{lem-antecompo2}
There exists $N_{0}$ such that for every $x\in \S_{i,\eps_{0}}$, there exists $x'\in \s^{-N_{0}}(x)$ and $x'\notin \S_{i,\eps_{0}}$.  
\end{lemma}
\begin{proof}
If $x=x_{0}x_{1}\ldots$ is a point in $\S_{i,\eps_{0}}$ and $[\om]=[\om_{0}\ldots \om_{k}]$ a cylinder with empty intersection with $B(\Om_{i},\eps_{0})$. Then mixing shows that there exists $z_{0}\ldots z_{r}$ with $z_{0}=\om_{k}$ and $z_{r}=x_{0}$ such that the word $\om_{0}\ldots\om_{k}z_{1}\ldots z_{r-1}x_{0}$ is admissible in $X$. Setting  
$$y:=\om_{0}\ldots\om_{k}z_{1}\ldots z_{r-1}x,$$
we get $\s^{k+r}(y)=x$ and $y\notin \S_{i,\eps_{0}}$. 

Note that, since $\S_{i,\eps_{0}}$ is invariant, no preimage of $y$ can be in $\S_{i,\eps_{0}}$. We also emphasize that $r$ does only depends on $x_{0}$, and the construction is rigid\footnote{in the sense it holds for any point sufficiently close to $x$.}.

Doing the same for any $[x_{0}]$ appearing in $\S_{i,\eps_{0}}$, we get $N_{0}$. 
\end{proof}

%%%%%
%%%%%

\subsection{Computations}
The key result is the following:
\begin{proposition}
\label{prop-minoprincipale}
For any irreducible component $\Om_{j}$, 
$$\ga+V(\Om_{i})\ge S^{ext}(j,i)+V(j)$$
\end{proposition}

The rest of the subsection is devoted to the proof of the proposition. 

\subsubsection{First step: an inequality involving $\Om_{i}$, $N_{0}$ and $\nu_{1/\infty}$}

%\begin{proposition} 
%Any accumulation point for $ \frac{1}{\beta}\log(P(\beta A) - h )$ belongs to $]-\8,0]$. Furthermore, if $\displaystyle{\lim_{\beta_j\to\infty} \frac{1}{\beta_j}\log(P(\beta_j A) - h )}=\ga$ and $\disp\lim_{\beta_j\to\infty} \frac{1}{\beta_{j}}\log(H_{\beta_{j}A})=V$, 
%then, 
% \begin{equation}\label{eq - gamma and V} \gamma + V(\Sigma_i) = \max_{j\in\{1,...,L\}}V(\Sigma_j)+a_{ij},\,\,\,\forall\,i\in\{1,...,k\}\end{equation} 
% and
%	\begin{equation}\label{eq - gamma and V - outra} V(\Sigma_i) = \max_{j\in\{1,...,L\}}V(\Sigma_j)+a_{ij},\,\,\,\forall\,i\in\{k+1,...,L\}.\end{equation} 
%\end{proposition} 

%\begin{proof} 

%As it is said above (see Lemma \ref{pressureselect}), $\be\mapsto P(\be A)-h$ is non-increasing  and goes to 0 as $\be\to+\8$. Hence, $\be\mapsto\disp \frac1\be\log(P(\be A)-h)$ is also non-increasing  and it is negative for sufficiently large $\be$. This shows that any accumulation point for $\disp\frac1\be\log(P(\be A)-h)$ is in $[-\8,0]$.  
%
%Because $V$ is Lipschitz continuous on $\Om$, thus bounded, and beause all the $a_{ij}$ are real numbers, Equation \eqref{eq - gamma and V} shows that $\ga$ is a real number, hence $\ga>-\8$. 

\bigskip
Let us now consider the Ruelle operator for $\beta A$, given by $${L}_{\beta A}(\varphi)(x) = \sum_{y\in \sigma^{-1}(x)}e^{\beta A(y)}\varphi(y).$$ 
Since $A_{\Om_{i}}\equiv0$,  $v(\eps):=\sup_{x\in \S_{i,\eps}}|A(x)|$ goes to $0$ as $\eps$ goes to 0.

We can also consider the Ruelle operator for the zero potential acting over $\CC(\Sigma_{i,\eps})$:
$$\mathcal{L}_{i,\eps}(\varphi)(x) = \sum_{y\in \sigma^{-1}(x),\,y\in \Sigma_{i,\eps}}\varphi(y)$$ where $x\in \Sigma_{i,\eps}$.
Furthermore for the conformal measure\footnote{ or any convex combination of the conformal measures on each irreducible piece if $\S_{i,\eps}$ is not irreducible.} $\nu_{\eps}$  satisfies 
\begin{equation}
\label{eq1-entropipi}
e^{h_{i,\eps}} \int \varphi(x) d\nu_{\eps}(x) = \int \mathcal{L}_{i,\eps}(\varphi)(x) \,d\nu_{\eps}(x),
\end{equation}
where $h_{i,\eps}$ is the topological entropy of $\Sigma_{i,\eps}$, which corresponds to the pressure of the zero function on $\Sigma_{i,\eps}$.

\bigskip
For $\eps<\eps_{0}$, $B(\Om_{i},\eps)\subset B(\Om_{i},\eps_{0})$ and then $\S_{i,\eps}\subset \S_{i,\eps_{0}}$. Therefore, Lemma \ref{lem-antecompo2} yields that  any point in $\S_{i,\eps}$ contains an $N_{0}$-preimage outside $B(\Om_{i,\eps_{0}})$. 

For $x\in \Sigma_{i,\eps}$ we have
	$$e^{N_{0}P(\beta A)}H_{\beta A}(x) = \sum_{y\in \sigma^{-N_{0}}(x),\,y\in \Sigma_{i,\eps}}e^{\beta S_{N_{0}}(A)(y)}H_{\be A}(y) +  \sum_{y\in \sigma^{-N_{0}}(x),\,y\notin \Sigma_{i,\eps}}e^{\beta S_{N_{0}}(A)(y)}H_{\be A}(y).$$

Let $r(\eps)$ denote $\sup_{z\in B(\Om_{i},\eps)}|A(z)|$. Note that $A_{|\Om_{i}}\equiv 0$ (by Lemma \ref{lem-AonOm}), thus $r(\eps)$ goes to $0$ as $\eps$ goes to 0. Then we get 
$$e^{N_{0}P(\beta A)}H_{\beta A}(x) \ge e^{-\be. N_{0}r(\eps)}\sum_{y\in \sigma^{-N_{0}}(x),\,y\in \Sigma_{i,\eps}}H_{\be A}(y) +  \sum_{y\in \sigma^{-N_{0}}(x),\,y\notin \Sigma_{i,\eps}}e^{\beta S_{N_{0}}(A)(y)}H_{\be A}(y).$$
	
This later inequality can be rewritten as

	$$e^{n_{0}P(\beta A)}H_{\beta A}(x)\ge e^{-\be.N_{0}r(\eps)}\mathcal{L}^{N_{0}}_{i,\eps}(H_{\be A})(x) +  \sum_{y\in \sigma^{-N_{0}}(x),\,y\notin \Sigma_{i,\eps}}e^{\beta S_{N_{0}}(A)(y)}H_{\be A}(y).$$
	Integrating both sides with respect to the eigenmeasure $\nu_{\eps}$ of $\mathcal{L}_{i,\eps}$ we have\footnote{if $\S_{i,\eps}$ has several irreducible components, the computation can be done on each of these components.}
	$$\hskip -2cme^{N_{0.}P(\beta A)}\int H_{\beta A}(x)\,d\nu_{\eps}(x) \ge e^{-\be.N_{0}r(\eps)} \int \mathcal{L}^{N_{0}}_{i,\eps}(H_{\be A})(x)\,d\nu_{\eps}(x) +  \int\sum_{y\in \sigma^{-N_{0}}(x),\,y\notin \Sigma_{i,\eps}}e^{\beta S_{N_{0}}(A)(y)}H_{\be A}(y)\,d\nu_{\eps}(x)$$
	which yields 
	$$\hskip -2cm e^{N_{0}P(\beta A)}\int H_{\beta A}(x)\,d\nu_{\eps}(x) \ge e^{N_{0}h_{i,\eps}-\be.N_{0}r(\eps)}\int H_{\be A}(x)\,d\nu_{\eps}(x) +  \int\sum_{y\in \sigma^{-N_{0}}(x),\,y\notin \Sigma_{i,\eps}}e^{\beta S_{N_{0}}(A)(y)}H_{\be A}(y)\,d\nu_{\eps}(x).$$
	
If $z\notin B(\Om_{i},\eps_{0})$, then $z\notin \S_{i,\eps}$, and therefore we get

 \begin{equation}
\label{equja1-nono}
\hskip -2cme^{N_{0}P(\beta A)}\int H_{\beta A}(x)\,d\nu_{\eps}(x) \ge e^{N_{0}h_{i,\eps}-\be.N_{0}r(\eps)}\int H_{\be A}(x)\,d\nu_{\eps}(x) +  \int\sum_{y\in \sigma^{-N_{0}}(x),\,y\notin B(\Om_{i},\eps_{0})}e^{\beta S_{N_{0}}(A)(y)}H_{\be A}(y)\,d\nu_{\eps}(x).
\end{equation}

	Note that $h_{i,\eps}$ decreases as $\eps$  decreases. Every $\S_{i,\eps}$ contains $\Om_{i}$, hence $h_{i,\eps}\ge h_{i}$ for every $\eps$. On the other side, if $\mu_{\eps}$ is a measure of maximal entropy for $\S_{i,\eps}$, its support  is contained in $B(\Om_{i},\eps)$, thus any accumulation point has support in $\Om_{i}$ (see Lemma \ref{lem-accunueps}). Consequently, the limit for $h_{i,\eps}$ as $\eps$ goes to 0 is lower or equal to $h_{i}$. This yields $\disp h_{i,\eps}\to_{\eps\to0}h_{i}$.

Then, considering any accumulation point $\nu_{1/\8}$ for $\nu_{\eps}$ for the weak* topology, \eqref{equja1-nono} yields 
 \begin{equation}
\label{equja2-nono}
\hskip -2cme^{N_{0}P(\beta A)}\int H_{\beta A}(x)\,d\nu_{1/\8}(x) \ge e^{N_{0}h_{i}}\int H_{\be A}(x)\,d\nu_{1/\8}(x) +  \int\sum_{y\in \sigma^{-N_{0}}(x),\,y\notin B(\Om_{i},\eps_{0})}e^{\beta S_{N_{0}}(A)(y)}H_{\be A}(y)\,d\nu_{1/\8}(x).
\end{equation}

or equivalently, 
 \begin{equation}
\label{equja3-nono}
\hskip -2cm (e^{N_{0}(P(\beta A)-h_{i})}-1)e^{N_{0}h_{i}}\int H_{\beta A}(x)\,d\nu_{1/\8}(x) \ge  \int\sum_{y\in \sigma^{-N_{0}}(x),\,y\notin B(\Om_{i},\eps_{0})}e^{\beta S_{N_{0}}(A)(y)}H_{\be A}(y)d\nu_{1/\8}(x).
\end{equation}	

\subsubsection{Limit for the left hand side term at exponential scale}

If we do $\dfrac1\be\log$ on the left hand side in \eqref{equja3-nono}
 and then do $\be\to+\8$, then 
 \begin{itemize}
\item $\dfrac1\beta\log (e^{N_{0}(P(\beta A)-h_{i})}-1)$ goes to $\ga$ by assumption on ``$\be\to+\8$'',
\item $\dfrac1\beta\log (e^{N_{0}h_{i}})$ goes to 0 and 
\item $\disp\dfrac1\beta\log \left(\int H_{\beta A}(x)\,d\nu_{1/\8}(x) \right)$ goes to $V(\Om_{i})$ since $\dfrac1\be\log H_{\be} $ uniformly (in $X$) goes to $V$ and $\nu_{1/\8}$ has support in $\Om_{i}$. 
\end{itemize}

%\begin{remark}
%\label{rem-entroppamax}
%In case of $h_{i}<h$, the quantity $\dfrac1\beta\log (e^{N_{0}(P(\beta A)-h_{i})}-1)$ goes to zero. In that case the inequality final becomes 
%$$V(\Om_{i})\ge S^{ext}(j,i)+V(\Om_{j}).$$
%$\blacksquare$\end{remark}
% 

\subsubsection{Limit for the right hand side term at exponential scale in the case $j\neq i$}\label{sssec-jneqi}
Inequality \eqref{equja3-nono} holds for any choice of $\eps_{0}$. We can thus choose $\eps_{0}$ such that $d(\Om_{j},\Om_{i})$ is larger than $10.\eps_{0}$. 

Let $x$ be any point in $\Om_{i}$. 
 Lemma \ref{lem-laplace} yields\footnote{We emphasize we use a weaker version of the Lemma, with a constant point $x$.} 
\begin{equation}
\label{eq-calcul1-1}
\lim_{\be\to+\8}\frac1\be\log\left(\sum_{y\in \sigma^{-N_{0}}(x),\,y\notin B(\Om_{i},\eps_{0})}e^{\beta S_{N_{0}}(A)(y)}H_{\be A}(y)\right)=\max_{y,\ \s^{N_{0}}(y)=x,\, y\notin B(\Om_{i},\eps_{0})}\{S_{N_{0}}(A)(y)+V(y)\}.
\end{equation}

Let us pick such a $y$. By definition, $y\notin \Om_{i}$ but $\s^{N_{0}}(y)=x$ belongs to $\Om_{i}$. Therefore one can consider the smallest $1\le k_{y}\le N_{0}$ such that $x':=\s^{k_{y}}(y)$ belongs to $\Om_{i}$. We denote this operation by $y\rightsquigarrow x'$.

\begin{remark}
\label{rem-ky}
In the following we shall always use $k_{y}$ but we let the reader check that $k_{y}$ actually only depends on $x'$.
$\blacksquare$\end{remark}

We get 
$$S_{N_{0}}(A)(z)=S_{k}(A)(y)+S_{N_{0}-k_{y}}(A)(x')=S_{k_{y}}(A)(y),$$
since $A\equiv 0$ on $\Om_{i}$.  
Furthermore Lemma \ref{lem-mane} yields  $V(y)\ge V(\Om_{j})+S(\Om_{j},y),$ which yields 
\begin{equation}
\label{eq-minosydney1}
S_{N_{0}}(A)(y)+V(y)\ge V(\Om_{j})+S(\Om_{j},y)+S_{k_{y}}(A)(y),
\end{equation}
and $\s^{k_{y}}(y)=x'\in \Om_{i}$.

If $z$ is a preimage of $x'$ and $z\notin \Om_{i}$, then $z$ is also a preimage of $x$ and any preimage of $z$ of order $k_{y}-1$ is a preimage of $x$ of  order $N_{0}$ and is outside $\Om_{i}$. 
By construction, if $y'$ is one of these preimages, we also get $y'\rightsquigarrow x'$.

We also remind the immediate inequality:
\begin{equation}
\label{equ-minosydneypm1}
\max_{y,\ \s^{N_{0}}(y)=x,\ y\notin\Om_{i}}\left\{S(\Om_{j},y)+S_{N_{0}}(y)\right\}\geq \max\left\{S(\Om_{j},y)+S_{k_{y}}(A)(y), \ y\rightsquigarrow x', \s^{k_{y}-1}(y)\notin \Om_{i}\right\}.
\end{equation}
Note that $y\rightsquigarrow x'$ and $\s^{k_{y}-1}(y)\notin \Om_{i}$ exactly means $\s^{N_{0}}(y)=x$, $\s^{k_{y}}(y)=x'$ and $\s^{k_{y}-1}(y')\notin \Om_{i}$. 
Therefore \eqref{equ-minosydneypm1} is equivalent to

\begin{eqnarray}
\max_{y,\ \s^{N_{0}}(y)=x,\ y\notin\Om_{i}}\left\{S(\Om_{j},y)+S_{N_{0}}(y)\right\}\geq&&\nonumber\\
 \max\left\{S(\Om_{j},y)+S_{k_{y}-1}(A)(y)+A(z), \ \s(z)=x',z\notin \Om_{i}, y\in \s^{-k_{y}+1}(z)\right\}.&&\label{equ-minosydneypm2}
\end{eqnarray}

Now, let us pick some $z$ satisfying $z\notin\Om_{i}$ and $\s(z)=x'$. 
We consider a sequence $(\xi^{n})$ of points such as in Lemma \ref{lem-bonnesuite} such that $S(\Om_{j},z)=\lim_{\ninf} S_{k_{n}}(A)(\xi^{n})$ holds. 
As $\s^{-k_{y}+1}(\{z\})$ is finite, infinitely manies $\xi^{n}$ satisfy $\s^{k_{n}-k_{y}}(\xi^{n})=y$ where $y$ is some preimage of order $k_{y}-1$ of $z$. 
Furthermore, for such $\xi^{n}$, Lemma \ref{lem-techmane} yields\footnote{Note that $k_{n}\to+\8$ since $x\notin\Om_{j}$ and $\xi^{n}$ converges to some point in $\Om_{j}$.}
$$S(\Om_{j},y)\geq \limsup_{\ninf}S_{k_{n}-k_{y}}(A)(\xi^{n}).$$
This yields that $ \max\left\{S(\Om_{j},y)+S_{k_{y}-1}(A)(y)+A(z), \ \s(z)=x',z\notin \Om_{i}, y\in \s^{-k_{y}+1}(z)\right\}$ is greater or equal to 
$ \max\left\{S(\Om_{j},z))+A(z), \ \s(z)=x',z\notin \Om_{i}, \right\}$ which is greater or equal to $S^{ext}(j,i)$.

This actually yields
$$\ga+V(\Om_{i})\ge S^{ext}(j,i)+V(\Om_{i}).$$

\subsubsection{Proof for the case $j=i$}
For the case $j\neq i$, one of the key points in the computation  of $S^{ext}(j,i)$ was that relevant backward orbits have to be uniformly far from $\Om_{i}$ (that is at distance greater than $\eps_{0}$). This does not hold anymore for the case $j=i$. However, relevant backward orbits have to get out $\Om_{i}$, hence to be at positive distance of $\Om_{i}$. This is the fact we want to use to get the proof of Proposition \ref{prop-minoprincipale}.

\bigskip
We first emphasize a by-product result from \ref{sssec-jneqi}:
\begin{claim}
 \label{claimbypro}
 Any $x\in\Om_{i}$ admits a preimage $x'\in \Om_{i}$ of order less than $N_{0}-1$ such that $x'$ admits preimages of order 1 outside $\Om_{i}$.
\end{claim}

We can thus consider the set $\Om_{i,k,\eps}$ of points $x\in \Om_{i}$ such that:
\begin{enumerate}
\item $\exists x'\in \Om_{i},\ \s^{k}(x')=x,\ k\le N_{0}-1$,
\item for all these $x'$'s, $\exists z\in\s^{-1}(x'), \ z\notin \Om_{i}$, 
\item for all these $z$'s, $d(z,\Om_{i})>\eps$. 
\end{enumerate}
Letting $\eps\downarrow 0$, one can find $\eps_{1}$ and $k=N_{1}-1\le N_{0}-1$ such that $\Om_{i,k,\eps}$ has positive $\nu_{1/\8}$ measure.

\bigskip
Then, one can redo what has been done above but replacing  $(\eps_{0},N_{0})$ with $(\eps_{1},N_{1})$. This yields
\begin{equation}
\label{equja4-nono}
\hskip -2cm (e^{N_{1}(P(\beta A)-h_{i})}-1)e^{N_{1}h_{i}}\int H_{\beta A}(x)\,d\nu_{1/\8}(x) \ge  \int\sum_{y\in \sigma^{-N_{1}}(x),\,y\notin B(\Om_{i},\eps_{1})}e^{\beta S_{N_{1}}(A)(y)}H_{\be A}(y)d\nu_{1/\8}(x).
\end{equation}	

The result for the left hand holds as previously. We thus just need to study the right hand side term. 
Note that we get 
$$\hskip -2cm\int\sum_{y\in \sigma^{-N_{1}}(x),\,y\notin B(\Om_{i},\eps_{1})}e^{\beta S_{N_{1}}(A)(y)}H_{\be A}(y)d\nu_{1/\8}(x)\ge \int_{\Om_{i,N_{1}-1,\eps_{1}}}\sum_{y\in \sigma^{-N_{1}}(x),\,y\notin B(\Om_{i},\eps_{1})}e^{\beta S_{N_{1}}(A)(y)}H_{\be A}(y)d\nu_{1/\8}(x)$$
and this last term can we rewritten as 
$$\int_{\Om_{i,N_{1}-1,\eps_{1}}}\sum_{z}e^{\beta A(z)}H_{\be A}(z)d\nu_{1/\8}(x),$$
where we reuse notations from the set $\Om_{i,N_{1}-1,\eps_{1}}$. 
We emphasize that for these $z$'s $S_{N_{1}}(A)(z)=A(z)$ since $\s(z)=x'\in \Om_{i}$. By construction they satisfy $d(z',\Om_{i})>\eps_{1}$. 
For all of these $x$, 
$$\dfrac1\be\log\left(\sum_{z}e^{\beta A(z)}H_{\be A}(z)\right)\to_{\be\to+\8}\max_{z}\{A(z)+V(z)\},$$
and $V(z)\ge V(\Om_{i})+S(\Om_{i},z)$. This yields, 
$$\lim_{\be\to+\8}\dfrac1\be\log\left(\sum_{z}e^{\beta A(z)}H_{\be A}(z)\right)\ge \max_{z}\{S(\Om_{i},z)+A(z)\}+V(\Om_{i}).$$

As $\Om_{i,N_{1}-1,\eps_{1}}$ has positive $\nu_{1/\8}$-measure, 
\begin{eqnarray*}
\liminf_{\be\to+\8}\dfrac1\be\log\left(\int_{\Om_{i,N_{1}-1,\eps_{1}}}\sum_{z}e^{\beta A(z)}H_{\be A}(z)d\nu_{1/\8}\right)\ge \inf_{x\in \Om_{i,N_{1}-1},\eps_{1}}\max\{S(\Om_{i},z)+A(z),\ \s(z)=x'\}\\
+V(\Om_{i}).
\end{eqnarray*}
Now, $ \inf_{x\in \Om_{i,N_{1}-1,\eps_{1}}}\ldots $ is bigger than $ \inf_{x\in \Om_{i}}\ldots$, and the result holds.

%%%%%%%
%%%%%
\subsection{End of the proof of Theorem \ref{theoA}}

Let us pick $n$ different irreducible components with maximal entropy $h$, $\Om_{i_{j}}$, $1\le j\le n$, with $i_{j}\neq i_{j'}$ if $j\neq j'$. 
Proposition \ref{prop-minoprincipale} yields
\begin{eqnarray*}
V(\Om_{i_{1}})+\ga&\ge & S^{ext}(i_{2},i_{1})+V(\Om_{i_{2}}),\\
V(\Om_{i_{2}})+\ga&\ge & S^{ext}(i_{3},i_{2})+V(\Om_{i_{3}}),\\
&\vdots&\\
V(\Om_{i_{n}})+\ga&\ge & S^{ext}(i_{1},i_{n})+V(\Om_{i_{1}}).\\
\end{eqnarray*}
Doing the sum, we get 
$$\ga\ge \dfrac{S^{ext}(i_{2},i_{1})+S^{ext}(i_{3},i_{2})+\ldots+S^{ext}(i_{1},i_{n})}n.$$
This holds for any finite sub-family obtained from $\{i_{1},\ldots i_{n}\}$, which exactly means that $\ga$ is bigger than the eigenvalue for the matrix whose entries are the $S^{ext}({i_{k},i_{l}})$'s. 

This holds for any accumulation point $\ga$ for $\dfrac1\be\log(P(\be)-h)$, which concludes the proof of Theorem \ref{theoA}. 

\bigskip
Note that we immediatly get:

\begin{maintheorem}\label{theoB}
With the same assumptions, if $\Om$ has only one irreducible component with maximal entropy, $\Om'$, then
$$\liminf_{\be\to+\8}\dfrac1\be\log(P(\be)-h)\geq S^{ext}(\Om',\Om').$$
\end{maintheorem}

Note that by Lemma \ref{lem-siinul}, $S^{ext}(\Om',\Om')$ is a negative real number.

\end{document}